\documentclass{gCOV2e}

\usepackage{mathtools}
\usepackage{amsmath,amssymb}
\DeclareMathOperator{\RE}{Re} 
\allowdisplaybreaks
\begin{document}
\doi{10.1080/0278107YYxxxxxxxx}
 \issn{1563-5066}
\issnp{0278-1077}
\jvol{00} \jnum{00} \jyear{2012} \jmonth{December}
\markboth{S.  Beig  and V. Ravichandran}{Complex Variables and Elliptic Equations}

\articletype{RESEARCH ARTICLE}

\title{Convexity in one direction of convolution and convex combinations of harmonic functions}

\author{Subzar Beig and V. Ravichandran$^{\ast}$\thanks{$^\ast$Corresponding author. Email: vravi68@gmail.com
\vspace{6pt}} \\\vspace{6pt}  {\em{Department of Mathematics, University of Delhi,
Delhi--110 007, India}}\\\vspace{6pt}\received{v3.0 released December 2012} }

\maketitle

\begin{abstract}We show that the convolution of the harmonic function $f=h+\bar{g}$, where $h(z)+\emph{e}^{-2\textit{i}\gamma}g(z)=z/(1-\emph{e}^{\textit{i}\gamma}z)$ having analytic  dilatation $\emph{e}^{\textit{i}\theta} z^n (0\leq\theta<2\pi)$, with the mapping  $f_{a,\alpha}=h_{a,\alpha}+\overline{g}_{a,\alpha}$, where $h_{a,\alpha}(z)=(z/(1+a)-\emph{e}^{\textit{i}\alpha}z^2/2)/(1-\emph{e}^{\textit{i}\alpha}z)^2$, $g_{a,\alpha}(z)=(a  \emph{e}^{2\textit{i}\alpha}z/(1+a)-\emph{e}^{3\textit{i}\alpha}z^2/2)/(1-\emph{e}^{\textit{i}\alpha}z)^2$  is convex in the direction  $-(\alpha+\gamma)$. We also show that the convolution of $f_{a,\alpha}$ with the right half-plane mapping having dilatation $(a-z^2)/(1-az^2)$ is convex in the direction  $-\alpha$. Finally, we introduce a family of univalent harmonic mappings and find out sufficient conditions for convexity along imaginary-axis of the linear combinations of harmonic functions of this family. \medskip

\medskip

\begin{classcode}Primary 31A05; Secondary 30C45\end{classcode}\medskip
\end{abstract}

\section{Introduction}
The complex-valued harmonic function $f$ on the unit disk $\mathbb{D}=\left\lbrace z\in\mathbb{C}:|z|<1\right\rbrace$ can be written as $f=h+\bar g$, where $h$ and $g$ are analytic functions  and are respectively known as analytic and co-analytic parts of $f$. By Lewy's theorem, the function $f$ is locally univalent and sense-preserving if and only if $h'(z)\neq0$ and the dilatation $\omega(z)=g'(z)/h'(z)$ is bounded by one on $\mathbb{D}$. Let $\mathcal{S}_H$ denote the class of all harmonic, sense-preserving and univalent mappings defined on $\mathbb{D}$ normalized by the conditions $f(0)=0$ and $f_z(0)=1$. Additionally, if the function $f$ satisfies $f_{\bar z}(0)=0$, then the class of such functions is denoted by $\mathcal{S}_H^0$. The sub-classes of $\mathcal{S}_H $ and  $\mathcal{S}_H^0$ consisting of functions mapping $\mathbb{D}$ onto convex domains are respectively denoted by $\mathcal{K}_H $ and  $\mathcal{K}_H^0$. For $0\leq\alpha<2\pi$, let $\mathcal{S}^0(H_\alpha)\subset\mathcal{S}_H^0 $ denote the class of all harmonic functions that maps $\mathbb{D} $ onto $H_\alpha:=\left\lbrace w \in \mathbb{C}:\RE (e^{\textit{i}\alpha}w)>-1/2\right\rbrace$. In \cite{1}, Dorff \textit{ et al.} showed that if $f=h+\bar{g}\in\mathcal{S}^0(H_\alpha),$ then\begin{equation}\label{p2eq1}
 h(z)+e^{-2\textit{i}\alpha}g(z)=\frac{z}{1-e^{\textit{i}\alpha}z }.
 \end{equation}
 A domain  $D$ is said to be \textit{convex in direction ${\theta}$} $(0\leq\theta<2\pi),$ if every line parallel to the line joining $0$ and  $\emph{e}^{\textit{i}\theta}$ lies completely inside or outside the domain $D$. If $\theta=0$, such a domain $D$ is called convex in the horizontal direction (CHD for short). The convolution (or Hadamard product) of two analytic functions $f$, $g:\mathbb{D}\rightarrow\mathbb{C}$ with Taylor series expansions \[f(z)=\sum_{n=1}^{\infty} a_n z^n \quad \text{ and }\quad g(z)=\sum_{n=1}^{\infty} b_n z^n,\] is defined  by $(f*g)(z)=\sum_{n=1}^\infty a_n b_nz^n$, and the harmonic convolution of the functions $f=h+\bar{g}$ and $F=H+\overline{G}$ is defined by $f*F= h*H+ \overline{g*G}.$ Consider the harmonic mapping $f_{a,\alpha}=h_{a,\alpha}+\overline{g}_{a,\alpha}$, $(-1<a<1,0\leq\alpha<2\pi)$, where
  \begin{equation}\label{p2eq2}h_{a,\alpha}(z)=\frac{z/(1+a)-\emph{e}^{\textit{i}\alpha}z^2/2}{(1-\emph{e}^{\textit{i}\alpha}z)^2} \quad \text{ and }\quad g_{a,\alpha}(z)=\frac{a  \emph{e}^{2\textit{i}\alpha}z/(1+a)-\emph{e}^{3\textit{i}\alpha}z^2/2}{(1-\emph{e}^{\textit{i}\alpha}z)^2}.
  \end{equation}
Also, we see from \eqref{p2eq2}  $f_{a,\alpha}(z)=\emph{e}^{-\textit{i}\alpha}f_{a,0}(\emph{e}^{\textit{i}\alpha}z)$. Therefore, $\RE\big(\emph{e}^{\textit{i}\alpha}f_{a,\alpha}(\emph{e}^{\textit{i}\alpha}z)\big)>-1/2$ and hence $f_{a,\alpha}\in\mathcal{S}^0(H_{\alpha})$, as $f_{a,0}$ is a right half-plane mapping.
  Therefore, \eqref{p2eq1} gives
\begin{equation}\label{p2eq3}
  h_{a,\alpha}(z)+\emph{e}^{-2\textit{i}\alpha}g_{a,\alpha}(z)=\frac{z}{1-\emph{e}^{\textit{i}\alpha}z}.
\end{equation}
The convolution of univalent convex harmonic function is not necessarily convex harmonic and it need not even  be univalent. Convexity in one direction of the convolution of mappings in the class $\mathcal{K}_H$ were studied in \cite{1,2,3,9}.
 \begin{lemma}\label{p2lem1}\cite[Theorem 2, p.491]{1}
Let $f_k\in\mathcal{S}^0(H_{\gamma_k}),(k=1,2)$. If $f_1*f_2$ is locally univalent and sense-preserving in $\mathbb{D}$, then $f_1*f_2\in\mathcal{S}_H^0$ and  is convex in the direction $-(\gamma_1+\gamma_2).$\end{lemma}
\begin{lemma}\label{p2lem2}\cite[Theorem 7, p.268]{9}
Let $f_1=h_1+\bar{g}_1$ is a right half-plane mapping given by $h_1+g_1=z/(1-z)$, and for $\pi/2\leq\alpha<\pi$, let $f_2=h_2+\bar{g}_2$ be a strip mapping given by $h_2(z)+g_2(z)=\frac{1}{2\textit{i}\sin\alpha}\log\left(\frac{1+\emph{e}^{\textit{i}\alpha}z}{1-\emph{e}^{\textit{i}\alpha}z}\right)$. If $f_1*f_2$ is locally univalent and sense-preserving, then $f_1*f_2\in\mathcal{S}_H^0$ and is convex in the direction of the real axis.
\end{lemma}
\begin{lemma}\label{p2lem3}\cite[Theorem 1.1]{2}
Let $f\in\mathcal{S}^0(H_{\gamma})$ with dilatation $\omega(z)=\emph{e}^{\textit{i}\theta}z^n$ $(n\in\mathbb{N},\theta\in\mathbb{R})$ and $f_{0,0}=h_{0,0}+\bar{g}_{0,0}$, where $h_{0,0}$, $g_{0,0}$ are  given by \eqref{p2eq2}. If $n=1,2$, then $f_{0,0}*f\in\mathcal{S}_H^0$ and is convex in the direction $-\gamma$.
\end{lemma}
\begin{lemma}\label{p2lem4}\cite[Theorem 2.2]{3}
Let $f_{a,0}=h_{a,0}+\bar{g}_{a,0}$, where $h_{a,\alpha}$, $g_{a,\alpha}$ are  given by \eqref{p2eq2}. If  $f=h+\bar{g}$ is a right half-plane mapping given by $h+g=z/(1-z)$ with dilatation $\omega(z)=\emph{e}^{\textit{i}\theta}z^n$ $(n\in\mathbb{N},\theta\in\mathbb{R})$, then $f_{a,0}*f\in\mathcal{S}_H^0$ and is CHD for $a\in[(n-2)/(n+2),1).$
\end{lemma}
In this paper, we generalize the result in Lemma \ref{p2lem4} by showing that the convolution $f_{a,\alpha}*f$ is convex in the direction $-(\gamma+\alpha)$ of the mappings $f_{a,\alpha}$ as given by \eqref{p2eq2} with $f=h+\bar{g}$, satisfying $ h(z)+e^{-2\textit{i}\gamma}g(z)={z}/{(1-e^{\textit{i}\gamma} z)}$ with the dilatation $\omega(z)=\emph{e}^{\textit{i}\theta}z^n$ $(n\in\mathbb{N},\theta\in\mathbb{R})$. We also find the values of $a$, for the  convolution of $f_{a,\alpha}$ with the right half-plane mapping having dilatations $(a-z^2)/(1-az^2)$ and $-(a-z)^2/(1-az)^2$ to be convex in the direction of $-\alpha$. Finally, we study convex combination of mappings from a family of locally-univalent and sense-preserving mappings $f=h+\bar{g}$ obtained by shearing of $h(z)+g(z)=\big( z(1+z^2)(1+z^4)...(1+z^{2^{n}}+\alpha z^{2^{n-1}})/(1+z^{2^{n+1}})\big)*\log({ 1/(1-z)}), n\in\mathbb{N},\alpha\in[-1,1]$, for different choices of the dilatation $\omega=g'/h'$.
\section{Main Results}
We begin this section with the following lemma, which gives a relation among the dilatations of harmonic mappings and their convolution.

\begin{lemma}\label{p2lem6}
Let the function $f_{a,0}=h_{a,0}+\overline{g}_{a,0}$ be harmonic mapping, where $h_{a,0}$, $g_{a,0}$ are given by \eqref{p2eq2}. If $\omega$ is the dilatation of slanted right half-plane mapping $f_{\gamma}=h_{\gamma}+\overline{g}_{\gamma}\in \mathcal{S}(H_{\gamma})$, then the dilatation $\tilde{\omega}$ of $f_{a,0}*f_{\gamma}$ is given by\begin{equation}\label{p2eq6}\tilde{\omega}(z)=\frac{2\omega(z)(1+\emph{e}^{-2\textit{i}\gamma} \omega(z))(a-\emph{e}^{\textit{i}\gamma}z)+z{\omega}'(z)(a-1)(1-\emph{e}^{\textit{i}\gamma}z)}{2(1-a\emph{e}^{\textit{i}\gamma}z)(1+\emph{e}^{-2\textit{i}\gamma}\omega(z) )+\emph{e}^{-2\textit{i}\gamma}z{\omega}'(z)(a-1)(1-\emph{e}^{\textit{i}\gamma}z)}.\end{equation}
\end{lemma}
\begin{proof}
Let the function $f=h+\bar{g}$ be a harmonic mapping with the dilatation $\omega_1=g'/h'$ and let $f_{a,0}*f=h_{a,0}*h+\overline{g_{a,0}*g}=:h_1+\overline{g}_1$. A calculation shows that\[h_1(z)=h_{a,0}(z)*h(z)=\frac{1}{2}\bigg(h(z)+\frac{1-a}{1+a}zh'(z)\bigg),\] and\[g_1(z)=g_{a,0}(z)*g(z)=\frac{1}{2}\bigg(g(z)-\frac{1-a}{1+a}zg'(z)\bigg),\]and the dilatation $\tilde{\omega}$ of $f_{a,0}*f$ is given by\begin{align}\tilde{\omega_1}(z)&=\frac{g'_1(z)}{h'_1(z)}=\frac{2ag'(z)-(1-a)zg''(z)}{2h'(z)+(1-a)zh''(z)}\notag\\&=\frac{2a\omega_1 h'(z)-(1-a)z(\omega_1(z) h''(z)+\omega'_1(z)h'(z)) }{2h'(z)+(1-a)zh''(z)}.\label{p2eq5}
\end{align}
Since, $f_{\gamma}=h_{\gamma}+\overline{g}_{\gamma}\in \mathcal{S}(H_{\gamma})$ and $\omega$ is its dilatation, we have\begin{equation}\label{p2eq5a}\omega(z) h'_{\gamma}(z)=g_{\gamma}'(z) \quad \text{ and }\quad h_{\gamma}(z)+\emph{e}^{-2\textit{i}\gamma}g_{\gamma}(z)=\frac{z}{1-\emph{e}^{\textit{i}\gamma}z}.
\end{equation}
The above two equations in \eqref{p2eq5a} together gives \begin{align}
h'_{\gamma}(z)&=\frac{1}{(1+\emph{e}^{-2\textit{i}\gamma} \omega(z))(1-\emph{e}^{\textit{i}\gamma}z)^2},\label{p2eq5b}\\ \intertext{and}h''_{\gamma}(z)&=\frac{2(1+\emph{e}^{-2\textit{i}\gamma} \omega(z))\emph{e}^{\textit{i}\gamma}-\emph{e}^{-2\textit{i}\gamma}{\omega}'(z)(1-\emph{e}^{\textit{i}\gamma}z)}{(1+\emph{e}^{-2\textit{i}\gamma} \omega(z))^2(1-\emph{e}^{\textit{i}\gamma}z)^3}.\label{p2eq5c}
\end{align}Now, using the expressions of $h'_\gamma$ and $h''_\gamma$ given by \eqref{p2eq5b} and \eqref{p2eq5c} in place of $h'$ and $h''$ in \eqref{p2eq5} and replacing $\omega_1$ by $\omega$, we get the desired expression for the dilatation  $\tilde{\omega}(z)$ of the convolution $f_{a,0}*f_{\gamma}$.
 \end{proof}
\begin{theorem}\label{p2theom2}
Let the function $f_{a,\alpha}=h_{a,\alpha}+\overline{g}_{a,\alpha}$ be harmonic mapping, where $h_{a,\alpha}$, $g_{a,\alpha}$ are given by \eqref{p2eq2}. If  $\omega=\emph{e}^{\textit{i}\theta}z^n(\theta\in\mathbb{R},n\in\mathbb{N})$ is the  dilatation of slanted right half-plane mapping $f_{\gamma}=h_{\gamma}+\overline{g}_{\gamma}\in \mathcal{S}(H_{\gamma})$, then the function $f_{a,\alpha}*f_{\gamma}\in\mathcal{S}_H^0$ and is convex in the direction $-(\alpha+\gamma)$ for $a\in[(n-2)/(n+2),1).$
\end{theorem}
\begin{proof}
Since\[h_{a,\alpha}(z)=\frac{z/(1+a)-\emph{e}^{\textit{i}\alpha}z^2/2}{(1-\emph{e}^{\textit{i}\alpha}z)^2} \quad \text{ and }\quad g_{a,\alpha}(z)=\frac{a  \emph{e}^{2\textit{i}\alpha}z/(1+a)-\emph{e}^{3\textit{i}\alpha}z^2/2}{(1-\emph{e}^{\textit{i}\alpha}z)^2},\]therefore we have
\begin{align*}
 (f_{\gamma}*f_{a,\alpha})(z)&=(h_{\gamma}*h_{a,\alpha})(z)+\overline{(g_{\gamma}*g_{a,\alpha})}(z)\\&=h_{\gamma}(z)*\bigg(\frac{z/(1+a)-\emph{e}^{\textit{i}\alpha}z^2/2}{(1-\emph{e}^{\textit{i}\alpha}z)^2}\bigg)\\&\quad+\overline{g_{\gamma}(z)*\bigg(\frac{a  \emph{e}^{2\textit{i}\alpha}z/(1+a)-\emph{e}^{3\textit{i}\alpha}z^2/2}{(1-\emph{e}^{\textit{i}\alpha}z)^2}\bigg)}\\&=h_{\gamma}(z)*\bigg(\emph{e}^{-\textit{i}\alpha}\frac{\emph{e}^{\textit{i}\alpha}z/(1+a)-(\emph{e}^{\textit{i}\alpha}z)^2/2}{(1-\emph{e}^{\textit{i}\alpha}z)^2}\bigg)\\&\quad+\overline{g_{\gamma}(z)*\bigg(\emph{e}^{\textit{i}\alpha}\frac{a  \emph{e}^{\textit{i}\alpha}z/(1+a)-(\emph{e}^{\textit{i}\alpha}z)^2/2}{(1-\emph{e}^{\textit{i}\alpha}z)^2}\bigg)}\\&=\emph{e}^{-\textit{i}\alpha}\bigg(h_{\gamma}(z)*h_{a,0}(z\emph{e}^{\textit{i}\alpha})+\overline{g_{\gamma}(z)*g_{a,0}(z\emph{e}^{\textit{i}\alpha})}\bigg)\\&=\emph{e}^{-\textit{i}\alpha}(h_{\gamma}*h_{a,0}+\overline{g_{\gamma}*g_{a,0}})(z\emph{e}^{\textit{i}\alpha})\\&=\emph{e}^{-\textit{i}\alpha}(f_{\gamma}*f_{a,0})(z\emph{e}^{\textit{i}\alpha}).
\end{align*}
Let $\omega_{\gamma,\alpha}$ be  the dilatation of the function $f_{\gamma}*f_{a,\alpha}$. Therefore, we have by  above equation \begin{equation}\label{p2eq5d}\omega_{\gamma,\alpha}(z)=\emph{e}^{2\textit{i}\alpha}\omega_{\gamma,0}(\emph{e}^{\textit{i}\alpha}z).
\end{equation}
In order to prove the result, by Lemma \ref{p2lem1}, we just need to show that $|\omega_{\gamma,\alpha}(z)|<1$. Equation \eqref{p2eq5d} shows it is enough to prove the result for $\alpha=0$, that is to show $|\omega_{\gamma,0}(z)|<1$.
Now, by Lemma \ref{p2lem6} we have \[\omega_{\gamma,0}(z)=\frac{2\omega(z)(1+\emph{e}^{-2\textit{i}\gamma} \omega(z))(a-\emph{e}^{\textit{i}\gamma}z)+z{\omega}'(z)(a-1)(1-\emph{e}^{\textit{i}\gamma}z)}{2(1-a\emph{e}^{\textit{i}\gamma}z)(1+\emph{e}^{-2\textit{i}\gamma}\omega(z) )+\emph{e}^{-2\textit{i}\gamma}z{\omega}'(z)(a-1)(1-\emph{e}^{\textit{i}\gamma}z)}.\]Substitute $\omega(z)=\emph{e}^{\textit{i}\theta}z^n$ in the above equation, we get \begin{align}
\omega_{\gamma,0}(z)&=\frac{2\emph{e}^{\textit{i}\theta}z^n(1+\emph{e}^{\textit{i}\theta}\emph{e}^{-2\textit{i}\gamma} z^n)(a-\emph{e}^{\textit{i}\gamma}z)+n\emph{e}^{\textit{i}\theta}z^n(a-1)(1-\emph{e}^{\textit{i}\gamma}z)}{2(1-a\emph{e}^{\textit{i}\gamma}z)(1+\emph{e}^{\textit{i}\theta}\emph{e}^{-2\textit{i}\gamma}z^n )+n\emph{e}^{-2\textit{i}\gamma}\emph{e}^{\textit{i}\theta}z^n(a-1)(1-\emph{e}^{\textit{i}\gamma}z)}\notag\\&=-\emph{e}^{2\textit{i}\theta}\emph{e}^{-\textit{i}\gamma}z^n\left(\frac{\splitfrac{ z^{n+1}-a\emph{e}^{-\textit{i}\gamma}z^n+1/2(2-n+an)\emph{e}^{-\textit{i}\theta}\emph{e}^{2\textit{i}\gamma}z}{+1/2(n-2a-an)\emph{e}^{-\textit{i}\theta}\emph{e}^{\textit{i}\gamma}}}{\splitfrac{1/2(n-2a-an)\emph{e}^{\textit{i}\theta}\emph{e}^{-\textit{i}\gamma}z^{n+1}}{+1/2(2-n+an)\emph{e}^{\textit{i}\theta}\emph{e}^{-2\textit{i}\gamma}z^n-a\emph{e}^{\textit{i}\gamma}z+1}}\right).\label{p2eq7}
\end{align}
Put $z\emph{e}^{\textit{i}\gamma}=w$ and $\theta-(n+2)\gamma=\beta$. Then by using  $\eqref{p2eq7}$,  we get
\begin{align}
\omega_{\gamma,0}(z)&=-\emph{e}^{2\textit{i}\gamma}\emph{e}^{2\textit{i}\beta}w^{n}\left(\frac{\splitfrac{w^{n+1}-aw^n+1/2(2-n+an)\emph{e}^{-\textit{i}\beta}w}{+1/2(n-2a-an)\emph{e}^{-\textit{i}\beta}}}{\splitfrac{1/2(n-2a-an)\emph{e}^{\textit{i}\beta}w^{n+1}}{+1/2(2-n+an)\emph{e}^{\textit{i}\beta}w^n-aw+1}}\right)\notag\\&=\emph{e}^{2\textit{i}\gamma}\omega_{0,0}(w),\label{eq7a}
\end{align}
where $\omega_{0,0}(w)$ corresponds to $\theta=\beta.$ By Lemma \ref{p2lem4}, $|\omega_{0,0}(w)|<1$ and hence \eqref{eq7a} gives $|\omega_{\gamma,0}(z)|<1$.
\end{proof}
In the next two theorems, we consider the right half-plane mapping with dilatations $(a-z^2)/(1-az^2)$ and $-(a-z)^2/(1-az)^2$, and examine its convolution properties with the mapping $f_{a,0}=h_{a,0}+\bar{g}_{a,0}$, where $h_{a,0}$, $g_{a,0}$ are given by \eqref{p2eq2}. The proof of these results requires the following lemma due to Cohn.
\begin{lemma}\label{p2lem7}(Cohn's rule)\cite{4}. Given a polynomial $t(z)=a_0+a_1z+...+a_nz^n$ of degree n, let \[t^*(z)=z^n\overline{t}(1/\bar{z})=\bar{a}_n+\bar{a}_{n-1}\bar{z}+...\bar{a}_0\bar{z}^n.\]Denote by $r$ and $s$, the number of zeros of $t(z)$ inside and on the unit circle $|z|=1$ respectively. If $|a_0| < |a_n|$, then\[t_1(z)=\frac{\bar{a}_nt(n)-a_0t^*(z)}{z}\] is of degree $n-1$ and has $r_1=r-1$ and $s_1=s$  number of zeros inside and on  the unit circle $|z|=1$ respectively.
\end{lemma}
 \begin{theorem}\label{p2theom3}
 Let the function $f=h+\overline{g}$  be the harmonic right hal-plane mapping with $h(z)+g(z)=z/(1-z)$, and the dilatation $\omega(z)=(a-z^2)/(1-az^2),a\in[0,1)$. If the function $f_{a,\alpha}=h_{a,\alpha}+\bar{g}_{a,\alpha}$ is harmonic right half-plane mapping, where $h_{a,\alpha}$, $\bar{g}_{a,\alpha}$ are given by \eqref{p2eq2}, then the function $f_{a,\alpha}*f\in\mathcal{S}_H^0$ and is convex in the direction of real-axis.
 \end{theorem}
\begin{proof}
 If $a=0$, the result follows from Lemma \ref{p2lem3}, so we consider the case $0<a<1$. By Lemma \ref{p2lem1}, we only need to show that the dilatation of the function $f_{a,\alpha}*f$ is bounded by one in $\mathbb{D}$. From \eqref{p2eq5d}, we see it is enough to prove the result for $\alpha=0$. Let $\tilde{\omega}$ be dilatation of the function $f_{a,0}*f\in\mathcal{S}_H^0$. Setting $\gamma=0$ and  $\omega(z)=(a-z^2)/(1-az^2)$ in Lemma \ref{p2lem6}, we get\begin{align}
 \tilde{\omega}(z)&=\frac{p(z)}{p^*(z)},\label{p27b}\\ \intertext{where}p(z)&=z^4+(1-a)z^3-(4a-a^2-1)z^2+a(a-1)z+a^2\notag\\ \intertext{and} p^*(z)&=z^4\overline{p(1/\overline{z})}.\notag
  \end{align}Therefore, if $z_1,z_2,z_3,z_4$ are zeros of $p(z)$, then $1/z_1,1/z_2,1/z_3,1/z_4$ are zeros of $p^*(z)$, and hence we can write \eqref{p27b} as\[\tilde{\omega}(z)=\frac{(z-z_1)(z-z_2)(z-z_3)(z-z_4)}{(1-\overline{z}_1z)(1-\overline{z}_2z)(1-\overline{z}_3z)(1-\overline{z}_4z)}.\]Thus, in order to show that $|\tilde{\omega}(z)|<1$, it is enough to show that $|z_i|<1,$ $i=1,2,3,4.$\\ Consider the polynomial $q_1$ given by
  \begin{align*}
  q_1(z)&=\frac{p(z)-a^2p^*(z)}{z}\\&=(1-a^4)z^3+(1-a)(1+a^3)z^2+(1+a^2-4a)(1-a^2)z+a(a^2-1)\\&=(1-a^2)p_1(z), \intertext{where}p_1(z)&=(1+a^2)z^3+(1-a+a^2)z^2+(1+a^2-4a)z-a.
\end{align*}
Since $|a|^2<1$, by Cohn's rule, the number of zeros of the polynomial $p_1$ in $\mathbb{D}$ is one less than that of the polynomial $p$.
Again, consider the polynomial $q_2$ given by
\begin{align*}
q_2(z)&=\frac{(1+a^2)p_1(z)+ap_1^*(z)}{z}\\&=(1+a^2+a^4)z^2+(1-2a^2+a^4)z+1-3a+a^2-3a^3+a^4\\&=:p_2(z).
\end{align*}
Since $|a|<|1+a^2|$, by Cohn's rule, the number of zeros of the polynomial $p_2$ in  $\mathbb{D}$ is one less than that of the polynomial $p_1$. Also, consider the polynomial $q_3$ given by
\begin{align*}
q_3(z)&=\frac{(1+a^2+a^4)p_2(z)-(1-3a+a^2-3a^3+a^4)p^*_2(z)}{z}\\&=3a(a-1)^2(1+a^2)[(2+a+2a^2)z+(1+a)^2]\\&=3a(a-1)^2(1+a^2)p_3(z),\\\intertext{where}p_3(z)&=(2+a+2a^2)z+(1+a)^2.
\end{align*}
 Since $|1-3a+a^2-3a^3+a^4|<|1+a^2+a^4|$ for $0< a<1)$, by Cohn's rule the number of zeros of the polynomial $p_3$ in $\mathbb{D}$ is one less than that of the polynomial $p_2$.
Finally, $z=-(1+a)^2/(2+a+2a^2)\in\mathbb{D}$ is zero of the polynomial $p_3$  for $0< a<1$. Therefore, it follows that all the  four zeros of  the polynomial $p$ lies in  $\mathbb{D}$, and hence $|\tilde{\omega}(z)|<1$.
\end{proof}
 \begin{theorem}\label{p2theom4}
 Let the function $f=h+\overline{g}$  be the harmonic right half-plane mapping with $h(z)+g(z)=z/(1-z)$, and the dilatation $\omega(z)=-(a-z)^2/(1-az)^2,a\in[0,1)$. If the function $f_{a,\alpha}=h_{a,\alpha}+\bar{g}_{a,\alpha}$ is harmonic right half-plane mapping, where $h_{a,\alpha}$, $\bar{g}_{a,\alpha}$ are given by \eqref{p2eq2}, then the function $f_{a,\alpha}*f\in\mathcal{S}_H^0$ and is convex in the direction of real-axis.
 \end{theorem}
 \begin{proof}
 If $a=0$, the result follows from Lemma \ref{p2lem3}, so we consider the case $0<a<1$. By Lemma \ref{p2lem1}, we only need to show that the dilatation of the function $f_{a,\alpha}*f$ is bounded by one in $\mathbb{D}$. From \eqref{p2eq5d}, we see it is enough to prove the result for $\alpha=0$. Let $\tilde{\omega}$ be the dilatation of the function $f_{a,0}*f\in\mathcal{S}_H^0$. Setting $\gamma=0$ and  $\omega(z)=-(a-z)^2/(1-az)^2$ in Lemma \ref{p2lem6}, we get
\begin{align}
 \tilde{\omega}(z)&=\frac{-(a-z)^3(1+z)+(a-z)z(a-1)(1-az)}{(1-az)^3(1+z)+(a-z)z(a-1)(1-z)}\notag\\&=\frac{z^4+(1-4a+a^2)z^3+(1-4a+4a^2-a^3)z^2+(-a+4a^2-a^3)z-a^3}{-a^3z^4+(-a+4a^2-a^3)z^3+(1-4a+4a^2-a^3)z^2+(1-4a+a^2)z+1}\notag\\&=\frac{p(z)}{p^*(z)},\label{p27c}\\\intertext{where}p(z)&=z^4+(1-4a+a^2)z^3+(1-4a+4a^2-a^3)z^2+(-a+4a^2-a^3)z-a^3\notag
\end{align}
and  $p^*(z)=z^4\overline{p}(1/\overline{z}).$ Therefore, if $z_1,z_2,z_3,z_4$ are zeros of $p(z)$, then $1/z_1,1/z_2,1/z_3,1/z_4$ are zeros of $p^*(z)$, and hence we   can write \eqref{p27c} as\[\tilde{\omega}(z)=\frac{(z-z_1)(z-z_2)(z-z_3)(z-z_4)}{(1-\overline{z}_1z)(1-\overline{z}_2z)(1-\overline{z}_3z)(1-\overline{z}_4z)}.\]We shall show that $|\tilde{\omega}(z)|<1$ by proving that  $z_i\in\mathbb{D}$ for $i=1,2,3,4.$\\ Consider the polynomial $q_1$ given by
  \begin{align*}
  q_1(z)&=\frac{p(z)+a^3p^*(z)}{z}\\&=(1-a^6)z^3+(1-4a+a^2-a^4+4a^5-a^6)z^2\\&\quad{}+(1-4a+4a^2-a^3)(1+a^3)z+(-a+4a^2-4a^4+a^5)\\&=(1-a^2)p_1(z),\\ \intertext{where}p_1(z)&=(1+a^2+a^4)z^3+(1-4a+2a^2-4a^3+a^4)z^2\\&\quad{}+(1-4a+5a^2-4a^3+a^4)z-(a-4a^2+a^3).
\end{align*}
Since $|a|^3<1$, by Cohn's rule the number of zeros of the polynomial $p_1$ in $\mathbb{D}$ is one less than that of the polynomial $p$.
Again, consider the polynomial $q_2$ given by
\begin{align*}
q_2(z)&=\big((1+a^2+a^4)p_1(z)+(a-4a^2+a^3)p_1^*(z)\big)/z\\&=(1+a^2+8a^3-15a^4+8a^5+a^6+a^8)z^2\\&\quad{}+(1-4a+a^2)(1-a+a^2)^2(1+3a+a^2)z\\&\quad{}+(1-3a-2a^2+11a^3-9a^4+11a^5-2a^6-3a^7+a^8)\\&=:p_2(z).
\end{align*}
Since $|a-4a^2+a^3|<|1+a^2+a^4|$ for $0<a<1$, by Cohn's rule the number of zeros of the polynomial $p_2$ in the the $\mathbb{D}$ is one less than that of the polynomial $p_1$.
Also, consider the polynomial $q_3$ given by
 \begin{align*}
q_3(z)&=\frac{1}{z}\big((1+a^2+8a^3-15a^4+8a^5+a^6+a^8)^2p_2(z)\\&\quad{}-(1-3a-2a^2+11a^3-9a^4+11a^5-2a^6-3a^7+a^8)^2p_2^*(z)\big)\\&=3a(a^2-1)^2(1+a^2+a^4)\big(
(2-a-4a^2+16a^3-4a^4-a^5+2a^6)z\\&\quad{}+(1-2a-8a^2+8a^3-8a^4-2a^5+a^6)\big)\\&=3a(a^2-1)^2(1+a^2+a^4)p_3(z),
\end{align*}
where \begin{align*}
p_3(z)&=(2-a-4a^2+16a^3-4a^4-a^5+2a^6)z\\&\quad{}+(1-2a-8a^2+8a^3-8a^4-2a^5+a^6).\end{align*}
For $0<a<1$, both $1+a^2+8a^3-15a^4+8a^5+a^6+a^8$ and $1-3a-2a^2+11a^3-9a^4+11a^5-2a^6-3a^7+a^8$  are positive, and the difference of the  $2nd$ term from the $1st$ term  is $3a(a^2-1)^2(1+a+a^2)$, which is also positive on the interval $(0,1)$. Therefore, $|1-3a-2a^2+11a^3-9a^4+11a^5-2a^6-3a^7+a^8|<|1+a^2+8a^3-15a^4+8a^5+a^6+a^8|$ for $0<a<1$, and hence by Cohn's rule the number of zeros of the polynomial $p_3$ in  $\mathbb{D}$ is one less than that of the polynomial $p_2$. Finally, $|1-2a-8a^2+8a^3-8a^4-2a^5+a^6|<|2-a-4a^2+16a^3-4a^4-a^5+2a^6|$ on $0<a<1$ and hence the zero $z=-(1-2a-8a^2+8a^3-8a^4-2a^5+a^6)/(2-a-4a^2+16a^3-4a^4-a^5+2a^6)$ of the polynomial $p_3(z)$ lies in $\mathbb{D}$. Therefore, all the four zeros of the polynomial $p$ lies in $\mathbb{D}$, and hence $|\tilde{\omega}(z)|<1$.
 \end{proof}
 In the next theorem, we examine the convexity along real axis of the convolution of  the mapping $f_{0,0}=h_{0,0}+\bar{g}_{0,0}$, where $h_{0,0}$, $g_{0,0}$ are given by \eqref{p2eq2}, with the strip mapping instead of right half-plane mapping.
 \begin{theorem}\label{p2theom5}
Let the function $f=h+\overline{g}$  be harmonic  mapping given by \[h(z)+g(z)=\frac{1}{2\textit{i}}\log\left(\frac{1+\textit{i}z}{1-\textit{i}z}\right),\] with the dilatation $\omega(z)=(a-z^2)/(1-az^2),a\in(-1,1)$. If the function $f_{0,0}=h_{0,0}+\bar{g}_{0,0}$ is  harmonic right-half plane mapping, where $h_{0,0}$, $g_{0,0}$ are given by \eqref{p2eq2}, then the function $f*f_{0,0}\in\mathcal{S}_H^0$ and is convex in the direction of real axis.
\end{theorem}
\begin{proof}
Since, we have $\omega(z)=g'(z)/h'(z)$ and  $h(z)+g(z)=\frac{1}{2\textit{i}}\log\left(\frac{1+\textit{i}z}{1-\textit{i}z}\right)$, it follows that
\begin{align*}
h'(z)&=\frac{1}{(1+\omega(z))(1+z^2)}\\ \intertext{and} h''(z)&=-\frac{2z(1+\omega(z))+\omega'(z)(1+z^2)}{(1+\omega(z))^2(1+z^2)^2}.
\end{align*}  Using the above expressions for $h'$ and $h''$, by  \eqref{p2eq5} the dilatation $\tilde{\omega}(z)$ of the function $f_{0,0}*f$ reduces to
 \[\tilde{\omega}(z)=-z\frac{\omega'(z)(1+z^2)-2z\omega(z)(1+\omega(z))}{2(1+\omega(z))-\omega'(z)z(1+z^2)}.\]Substituting $\omega(z)=(a-z^2)/(1-az^2)$ in above equation, we get $\tilde{\omega}(z)=z^2$, and hence $|\tilde{\omega}|<1$ on $\mathbb{D}$. The result now  follows from Lemma \ref{p2lem2}.
\end{proof}
\section{Linear Combination of Harmonic mappings.}Before going into the detail in this section, we first introduce a result due to Hengartner and Schober for checking the convexity of analytic functions in the direction of imaginary-axis, and a result due to Clunie and Sheil-small for constructing univalent harmonic mapping convex in given direction. These results will be of interest in this section.
\begin{lemma}\label{p2lem8}\cite[Theorem 1, p.304]{5}
Suppose $f$ is analytic and non-constant mapping in $\mathbb{D}$, then\[\RE\big((1-z^2)f'(z)\big)>0,\quad\text{  }\quad z\in\mathbb{D}\]if and only if
\begin{itemize}
\item[(1)] $f$ is univalent in $\mathbb{D}$
\item[(2)] $f$ is convex in the direction of imaginary axis, and
\item[(3)] there exists sequences $z'_n$ and $z''_n$ converging to $z=1$ and $z=-1$, respectively, such that\[\lim _{n\rightarrow\infty}\RE(f(z'_n))=\sup_{|z|<1}\RE(f(z)),\]
\[\lim _{n\rightarrow\infty}\RE(f(z''_n))=\inf_{|z|<1}\RE(f(z)).\]
\end{itemize}
\end{lemma}
\begin{lemma}\label{p2lem9}\cite{6}
A locally univalent harmonic mapping $f=h+\overline{g}$ on $\mathbb{D}$ is univalent mapping of $\mathbb{D}$ onto a domain convex in the direction of $\phi$ if and only if $h-\emph{e}^{2\textit{i}\phi}g$ is univalent analytic mapping of $\mathbb{D}$ onto a  domain convex in the direction of $\phi$.
\end{lemma}
Wang \textit{et al.} gave a sufficient condition of univalency for the convex combination $f_3=tf_1+(1-t)f_2,0\leq t\leq 1$ of two harmonic univalent functions $f_1$ and $f_2$. Indeed, they have proved the following:
\begin{theorem}\label{p2theom6}\cite[Theorem 3, p.455]{7}
If the function $f_i=h_i+\overline{g}_i\in\mathcal{S}_H$ satisfies $h_i(z)+g_i(z)=z/(1-z)$ for $i=1,2$, then the convex combination $f_3=tf_1+(1-t)f_2$, $0\leq t\leq 1$, is univalent and convex in the direction of real axis.
\end{theorem}
Kumar \textit{et al.} \cite{8} introduce a locally univalent and sense-preserving harmonic functions $f_{\alpha}=h_{\alpha}+\overline{g}_{\alpha}$  given by $h_{\alpha}(z)+g_{\alpha}(z)=z(1-\alpha z)/(1-z^2)$, $\alpha\in[-1,1]$, with the dilatation $\omega=g'_{\alpha}/h'_{\alpha}\in\mathbb{D}$, and proved the following:
\begin{theorem}\label{p2theom7}\cite[Theorem 2.7]{8}
For $i=1,2$, let the function $f_{\alpha_i}=h_{\alpha_i}+\overline{g}_{\alpha_i}$ be normalized harmonic mapping satisfying $h_{\alpha_i}(z)+g_{\alpha_i}(z)=z(1-{\alpha_i} z)/(1-z^2)$, ${\alpha_i} \in[-1,1]$. If $\omega_1(z)=-z$ and $\omega_2(z)=z$ are the  dilatations respectively of the mappings $f_{\alpha_1}$ and $f_{\alpha_2}$, then their convex combination $f=tf_{\alpha_1}+(1-t)f_{\alpha_2}$, $0\leq t\leq 1$, belongs $\mathcal{S}_H$ and is convex in the direction of imaginary axis provided $\alpha_1\geq\alpha_2.$
\end{theorem}
\begin{theorem}\label{p2theom8}\cite[Theorem 2.9]{8}
For $i=1,2$, let the function $f_{\alpha_i}=h_{\alpha_i}+\overline{g}_{\alpha_i}$ be normalized harmonic mapping satisfying $h_{\alpha_i}(z)+g_{\alpha_i}(z)=z(1-{\alpha_i} z)/(1-z^2)$, ${\alpha_i} \in[-1,1]$. Let $\omega_1(z)=-z$ and $\omega_2$ be the dilatations respectively of mappings $f_{\alpha_1}$ and $f_{\alpha_2}$, with $|\omega_2(z)|<1$. Let $f=tf_{\alpha_1}+(1-t)f_{\alpha_2}$, $0\leq t\leq 1$, be convex combination of $f_{\alpha_1}$ and $f_{\alpha_2}$. Then, we have\begin{itemize}
\item[(1)] If $\omega_2(z)=-z^2$ and $\alpha_1>\alpha_2,$ then $f$ is in $\mathcal{S}_H$ and is convex in the direction of imaginary axis.
\item[(2)]If $\omega_2(z)=z^2$ and $|\alpha_1|>|\alpha_2|,$ and $\alpha_1\alpha_2\geq0,$ then $f$ is in $\mathcal{S}_H$ and is convex in the direction of imaginary axis.
\end{itemize}
\end{theorem}
For $\alpha\in[-2(\sqrt{2}-1),2(\sqrt{2}-1)]$, $n\in\mathbb{N}$, we will introduce a  family of locally univalent and sense-preserving harmonic mappings  $f_{\alpha,n}=h_{\alpha,n}+\overline{g}_{\alpha,n}$, given by
\begin{equation}\label{p27d} h_{\alpha,n}(z)+g_{\alpha,n}(z)=\frac{z(1+z^2)(1+z^4)\dots(1+z^{2^n}+\alpha z^{2^{n-1}})}{1+z^{2^{n+1}}}*\log{\frac{1}{1-z}},\quad{}z\in\mathbb{D},
\end{equation}
with the  dilatation $\omega(z)=g'_{\alpha,n}(z)/h'_{\alpha,n}(z)\in\mathbb{D}$.
 In this section, we will study the convexity in the direction of real axis of convex combinations of mappings in this family. First, we check the convexity in the direction of real axis of the functions $f_{\alpha,n}$. Differentiating \eqref{p27d}, we get
\begin{equation}\label{p27e} h'_{\alpha,n}(z)+g'_{\alpha,n}(z)=\frac{(1+z^2)(1+z^4)\dots(1+z^{2^n}+\alpha z^{2^{n-1}})}{1+z^{2^{n+1}}}.\end{equation}
Now, upon putting $z^{n-1}=w$, and using \eqref{p27e}, we see
\begin{align*}
\RE\big((&1-z^2)(h'_{\alpha,n}(z)+g'_{\alpha,n}(z))\big)\\&=\RE\bigg(\frac{(1-z^{2^{n}})(1+z^{2^n}+\alpha z^{2^{n-1}})}{1+z^{2^{n+1}}}\bigg)\\&=\RE\bigg(\frac{(1-w^2)(1+w^2+\alpha w)}{1+w^4}\bigg)\\&=\RE\bigg(\frac{1-w^4+\alpha w-\alpha w^3}{1+w^4}\bigg)\\&=\frac{1-|w|^8+\alpha(1-|w|^6)\RE(w)-\alpha(1-|w|^2)\RE(w^3)}{|1+w^4|^2}\\&=(1-|w|^4)\bigg(\frac{1+|w|^2+|w|^4+|w|^6+\alpha(1+|w|^2+|w|^4)\RE(w)-\alpha\RE(w^3)}{|1+w^4|^2}\bigg)\\&>0,
\end{align*}
if
\begin{align*}
1+|w|^2+&|w|^4+|w|^6+\alpha(1+|w|^2+|w|^4)\RE(w)-\alpha\RE(w^3)\\&=\big((|w|^2+|w|^4)(1+\alpha)\RE(w)\big)+\big((\sqrt{2}-1)^2+|w|^6-\alpha\RE(w^3)\big)\\&\quad+\big(1-(\sqrt{2}-1)^2+\alpha\RE(w)\big)>0
\end{align*}
Since for $\alpha\in[-2(\sqrt{2}-1,2(\sqrt{2}-1)]$,  the first term in the above sum is non-negative, the second and the third terms are positive, therefore \[1+|w|^2+|w|^4+|w|^6+\alpha(1+|w|^2+|w|^4)\RE(w)-\alpha\RE(w^3)>0,\]and hence
\begin{equation}\label{p2eqa}
\RE\big((1-z^2)(h'_{\alpha,n}(z)+g'_{\alpha,n}(z)\big)>0.
\end{equation}
 Therefore, by Lemma \ref{p2lem8}, $h_{\alpha,n}+g_{\alpha,n}$ is analytic and convex in the direction of imaginary axis, and hence  Lemma \ref{p2lem9} implies that the function  $f_{\alpha,n}=h_{\alpha,n}+\overline{g}_{\alpha,n}\in\mathcal{S}_H$ and is convex in the direction of imaginary axis.\

 In the next theorem, we will show that, for the convex combination of the functions  $f_{\alpha,n}$ to be convex in the direction of imaginary axis, it is sufficient for this combination to be  local univalent  and sense-preserving.
\begin{theorem}\label{p2theom10}
For $i=1,2,$ let the function $f_{\alpha_i,n}=h_{\alpha_i,n}+\overline{g}_{\alpha_i,n}$ be  normalized harmonic mapping, satisfying $h_{\alpha_i,n}(z)+g_{\alpha_i,n}(z)=\big( z(1+z^2)(1+z^4)\dots(1+z^{2^n}+\alpha_i z^{2^{n-1}})/(1+z^{2{n+1}})\big)*\log{1/(1-z)}, \alpha_i\in[-2(\sqrt{2}-1),2(\sqrt{2}-1)],n\in\mathbb{N}$ and  $|g'_{\alpha_i,n}/h'_{\alpha_i,n}|<1$ in  $\mathbb{D}$. Then the convex combination $f=tf_{\alpha_1,n}+(1-t)f_{\alpha_2,m},0\leq t\leq 1$ is in $\mathcal{S}_H$ and is convex in the direction of of imaginary axis, provided $f$ is locally univalent and sense-preserving.
\end{theorem}
\begin{proof}
We have  $f=tf_{\alpha_1,n}+(1-t)f_{\alpha_2,n}=th_{\alpha_1,n}+(1-t)h_{\alpha_2,n}+\overline{tg_{\alpha_1,n}+(1-t)g_{\alpha_2,n}}$. Let $F=th_{\alpha_1,n}+(1-t)h_{\alpha_2,n}+tg_{\alpha_1,n}+(1-t)g_{\alpha_2,n}$,   and   $F_{\alpha_i,n}=h_{\alpha_i,n}+g_{\alpha_i,n}$, $i=1,2$. Now by using $\eqref{p2eqa}$, we have
\begin{align*}
\RE[(1-z^2)F'(z)]&=\RE[(1-z)^2(tF'_{\alpha_1,n}(z)+(1-t)F'_{\alpha_2,n}(z))]\\&=t\RE[(1-z)^2F'_{\alpha_1,n}(z)]+(1-t)\RE[(1-z)^2F'_{\alpha_2,n}(z)]>0.
\end{align*}
 Therefore, by Lemma \ref{p2lem8}, the function  $F=th_{\alpha_1,n}+(1-t)h_{\alpha_2,n}+tg_{\alpha_1,n}+(1-t)g_{\alpha_2,n}$ is analytic and convex in the direction of imaginary axis, and hence  Lemma \ref{p2lem9} shows that the function  $f=tf_{\alpha_1,n}+(1-t)f_{\alpha_2,n}\in\mathcal{S}_H$ and is convex in the direction of imaginary axis.
\end{proof}
\begin{lemma}\label{p2theom11}
For $i=1,2$ and $n\in\mathbb{N},$ let the function $f_{\alpha_i,n}=h_{\alpha_i,n}+\overline{g}_{\alpha_i,n}$ be the normalized harmonic mapping, such that $h_{\alpha_i,n}(z)+g_{\alpha_i,n}(z)=\big(z(1+z^2)(1+z^4)\dots(1+z^{2^n}+\alpha_i z^{2^{n-1}})/(1+z^{2{n+1}})\big)*\log{1/(1-z)}$, $\alpha_i\in[-2(\sqrt{2}-1),2(\sqrt{2}-1)] $ and $ \omega_i=g'_{\alpha_i,n}/h'_{\alpha_i,n},$ with $ |\omega_i(z)|<1$ in $\mathbb{D}$. Then for $n\geq m,$ the dilatation $\tilde{\omega}$ of the convex combination $f=tf_{\alpha_1,n}+(1-t)f_{\alpha_2,m}$, $0\leq t\leq 1$  is given by
\begin{equation}\label{p2eq9}
\tilde{\omega}(z)=\frac{p(z)}
{q(z)},
\end{equation}
where\begin{align*}
p(z)&={\omega_1(z)\big(1+z^{2^m}\big)\dots\big(1+z^{2^n}+\alpha_1 z^{2^{n-1}}\big)(1+\omega_2(z))\big(1+z^{2^{m+1}}\big)}\\&\quad{}+(1-t)\omega_2(z)\big(1+z^{2^{n+1}}\big)\big(1+z^{2^{m}}+\alpha_2 z^{2^{m-1}}\big)(1+\omega_1(z)),\\\intertext{and}q(z)&={\big(1+z^{2^m}\big)\dots\big(1+z^{2^n}+\alpha_1 z^{2^{n-1}}\big)(1+\omega_2(z))\big(1+z^{2^{m+1}}\big)}\\&\quad{}+(1-t)\big(1+z^{2^{n+1}}\big)\big(1+z^{2^{m}}+\alpha_2 z^{2^{m-1}}\big)(1+\omega_1(z)).
\end{align*}
\end{lemma}
\begin{proof}
As $f=t f_{\alpha_1,n}+(1-t)f_{\alpha_2,m}=th_{\alpha_1,n}+(1-t)h_{\alpha_2,m}+t\bar{g}_{\alpha_1,n}+(1-t)\bar{g}_{\alpha_2,m}$, $ \omega_1=g'_{\alpha_1,n}/h'_{\alpha_1,n}$ and  $ \omega_2=g'_{\alpha_2,m}/h'_{\alpha_2,m},$  therefore the dilatation $\tilde{\omega}$ of the function $f$ is given by
\begin{equation}\label{p2eq10}
\tilde{\omega}=\frac{tg'_{\alpha_1,n}+(1-t)g'_{\alpha_2,m}}{th'_{\alpha_1,n}+(1-t)h'_{\alpha_2,m}}=\frac{t\omega_1h'_{\alpha_1,n}+(1-t)\omega_2h'_{\alpha_2,m}}{th'_{\alpha_1,n}+(1-t)h'_{\alpha_2,m }}.
\end{equation}
Also, we have\[h_{\alpha_1,n}(z)+g_{\alpha_1,n}(z)=\frac{z(1+z^2)(1+z^4)\dots\big(1+z^{2^n}+\alpha_1 z^{2^{n-1}}\big)}{\big(1+z^{2^{n+1}}\big)}*\log{\frac{1}{(1-z)}}\]Differentiating the above equation, and using $ \omega_1(z)=g'_{\alpha_1,n}(z)/h'_{\alpha_1,n}(z)$, we get
\[h'_{\alpha_1,n}(z)=\frac{(1+z^2)(1+z^4)\dots\big(1+z^{2^n}+\alpha_1 z^{2^{n-1}})}{(1+\omega_1(z))\big(1+z^{2^{n+1}}\big)}.\]Similarly, we see \[h'_{\alpha_2,m}(z)=\frac{(1+z^2)(1+z^4)\dots\big(1+z^{2^m}+\alpha_2 z^{2^{m-1}}\big)}{(1+\omega_2(z))\big(1+z^{2^{m+1}}\big)}.\]Now, using the above expressions for $h'_{\alpha_1,n}$ and $h'_{\alpha_2,m}$ in \eqref{p2eq10}, we get the desired  result.
\end{proof}
 For $n=m$ in Lemma \ref{p2theom11}, \eqref{p2eq9} reduces to
\begin{equation}\label{p2eq11}
\tilde{\omega}=\frac{t\omega_1\big(1+z^{2^n}+\alpha_1 z^{2^{n-1}}\big)(1+\omega_2)+(1-t)\omega_2\big(1+z^{2^n}+\alpha_2 z^{2^{n-1}}\big)(1+\omega_1)}{t(1+\omega_2)\big(1+z^{2^n}+\alpha_1 z^{2^{n-1}}\big)+(1+\omega_1)(1-t)\big(1+z^{2^n}+\alpha_2 z^{2^{n-1}}\big)}.
\end{equation}
\begin{theorem}\label{p2theom12}
For $i=1,2$ and $n\in\mathbb{N}$, let the function $f_{i,n}=h_{i,n}+\overline{g}_{i,n}$ be the normalized harmonic mapping such that $h_{i,n}(z)+g_{i,n}(z)=\big(z(1+z^2)(1+z^4)\dots(1+z^{2^n}+\alpha z^{2^{n-1}})/(1+z^{2^{n+1}})\big)*\log{1/(1-z)},$   $\alpha\in[-2(\sqrt{2}-1),2(\sqrt{2}-1)] $ and having dilatation $ \omega_i=g'_{i,n}/h'_{i,n}.$ If $ |\omega_i(z)|<1$ in  $\mathbb{D}$, then  the  convex combination $f=tf_{1,n}+(1-t)f_{2,n}$, $0\leq t\leq 1$ belongs to $\mathcal{S}_H$ and is convex in the direction of imaginary axis.
\end{theorem}
\begin{proof}
In view of the Theorem \ref{p2theom10}, we only need to show that the function $f$ is locally univalent and sense-preserving. Let $\tilde{\omega}$ be the dilatation of the function $f$. Setting $\alpha_1=\alpha_2=\alpha$ in \eqref{p2eq11}, we get\[\tilde{\omega}(z)=\frac{t\omega_1(z)+(1-t)\omega_2(z)+\omega_1(z)\omega_2(z)}{1+t\omega_2(z)+(1-t)\omega_1(z)}.\]Therefore, from \cite[Theorem 3]{7}, we get $|\tilde{\omega}|<1$. Hence  the function $f$ is locally univalent and sense-preserving.
\end{proof}
In the next three theorems, we examine the convexity in the direction of imaginary axis of the convex combinations of functions $f_{\alpha,n}$, having different dilatations.
\begin{theorem}\label{p2theom13}
For $i=1,2$ and $n\in\mathbb{N},$ let the function $f_{\alpha_i,n}=h_{\alpha_i,n}+\overline{g}_{\alpha_i,n}$ be normalized harmonic mapping satisfying $h_{\alpha_i,n}(z)+g_{\alpha_i,n}(z)=\big(z(1+z^2)(1+z^4)\dots(1+z^{2^n}+\alpha_i z^{2^{n-1}})/(1+z^{2{n+1}})\big)*\log{1/(1-z)}$, $ \alpha_i\in[-2(\sqrt{2}-1),2(\sqrt{2}-1)] $. If $\omega_1(z)=-z^{2^{n-1}}$ and $\omega_2(z)=z^{2^{n-1}}$ are dilatations respectively of the functions $f_{\alpha_1,n}$ and $f_{\alpha_2,n}$, then the  convex combination $f=tf_{\alpha_1,n}+(1-t)f_{\alpha_2,n}$, $0\leq t\leq 1$ belongs to $\mathcal{S}_H$ and is convex in the direction of imaginary axis provided $\alpha_1\geq\alpha_2.$
\end{theorem}
\begin{proof}
In view of Theorem \ref{p2theom10}, we only need to show that the function $f$ is locally univalent and sense-preserving. Let $\tilde{\omega}$ be the dilatation of the function $f$. Using $\omega_1(z)=-z^{2^{n-1}}$ and $\omega_2(z)=z^{2^{n-1}}$ in \eqref{p2eq11}, we get\[\tilde{\omega}(z)=\left(\frac{\splitfrac{-tz^{2^{n-1}}\big(1+z^{2^n}+\alpha_1 z^{2^{n-1}}\big)\big(1+z^{2^{n-1}}\big)}{+(1-t)z^{2^{n-1}}\big(1+z^{2^n}+\alpha_2 z^{2^{n-1}}\big)\big(1-z^{2^{n-1}}\big)}}{\splitfrac{t\big(1+z^{2^{n-1}}\big)\big(1+z^{2^n}+\alpha_1 z^{2^{n-1}}\big)}{+(1-t)\big(1-z^{2^{n-1}}\big)\big(1+z^{2^n}+\alpha_2 z^{2^{n-1}}\big)}}\right).\] Put $z^{2^{n-1}}=w$ in above equation, we get\begin{align}
\tilde{\omega}(z)&=\frac{-tw(1+w^2+\alpha_1 w)(1+w)+(1-t)w(1+w^2+\alpha_2 w)(1-w)}{t(1+w)(1+w^2+\alpha_1 w)+(1-t)(1-w)(1+w+\alpha_2 w)}\notag\\&=-w\frac{w^3+(2t-1+\alpha_1 t+\alpha_2(1-t))w^2+(1-\alpha_2(1-t)+\alpha_1t)w+(2t-1)}{(2t-1)w^3+(1-\alpha_2(1-t)+\alpha_1t)w^2+(2t-1+\alpha_1 t+\alpha_2(1-t))w+1},\notag
\end{align}
 which is the dilatation $\omega(w)$ of the function $f$ in the Theorem \ref{p2theom7}, with $-2\alpha_i$ replaced by $\alpha_i$, see  \cite[Theorem 2.7]{8}. Therefore, $|\tilde{\omega}|<1$. Hence the function $f$ is locally univalent and sense-preserving.
\end{proof}
\begin{theorem}\label{p2theom14}
For $i=1,2$ and $n\in\mathbb{N},$ let the function $f_{\alpha_i,n}=h_{\alpha_i,n}+\overline{g}_{\alpha_i,n}$ be the normalized harmonic mapping satisfying $h_{\alpha,n}(z)+g_{\alpha,n}(z)=\big(z(1+z^2)(1+z^4)\dots(1+z^{2^n}+\alpha z^{2^{n-1}})/(1+z^{2{n+1}})\big)*\log{1/(1-z)}, \alpha_i\in[-2(\sqrt{2}-1),2(\sqrt{2}-1)] $. If $\omega_1=-z^{2^{n-1}}$ and $\omega_2(z)$ $(|\omega_2|<1)$ are dilatations respectively of $f_{\alpha_1,n}$ and $f_{\alpha_2,n}$, then for the   convex combination $f=tf_{\alpha_1,n}+(1-t)f_{\alpha_2,n},0\leq t\leq 1,$ we have
\begin{itemize}
\item[(1)]If $\omega_2(z)=-z^{2^n}$ and $\alpha_1>\alpha_2,$ then the function $f$ belongs to $\mathcal{S}_H$ and is convex in the direction of imaginary axis.
\item[(2)]If $\omega_2(z)=z^{2^n}$ and $|\alpha_1|>|\alpha_2|,$ and $\alpha_1\alpha_2\geq0,$ then the function $f$ belongs to $\mathcal{S}_H$ and is convex in the direction of imaginary axis.
\end{itemize}
\end{theorem}
\begin{proof}
In view of Theorem \ref{p2theom10}, we only need to show that the function $f$ is locally univalent and sense-preserving. Let $\tilde{\omega}$ be the  dilatation of the function $f$. Using $\omega_1(z)=-z^{2^{n-1}}$ and $\omega_2(z)=-z^{2^n}$ in \eqref{p2eq11}, we get\[\tilde{\omega}(z)=\left(\frac{\splitfrac{-tz^{2^{n-1}}(1+z^{2^n}+\alpha_1 z^{2^{n-1}})(1-z^{2^n})}{-(1-t)z^{2^n}(1+z^{2^n}+\alpha_2 z^{2^{n-1}})(1-z^{2^{n-1}})}}{\splitfrac{t(1-z^{2^n})(1+z^{2^n}+\alpha_1 z^{2^{n-1}})}{+(1-t)(1-z^{2^{n-1}})(1+z^{2^n}+\alpha_2 z^{2^{n-1}})}}\right).\] Put $z^{2^{n-1}}=w$ , above equation gives \[\tilde{\omega}(z)=\frac{-tw(1+w^2+\alpha_1 w)(1-w^2)-(1-t)w^2(1+w^2+\alpha_2 w)(1-w)}{t(1-w^2)(1+w^2+\alpha_1 w)+(1-t)(1-w)(1+w+\alpha_2 w)},\]
  which is the dilatation $\omega(w)$ of the function $f$ in the Theorem \ref{p2theom8}, with $-2\alpha_i$ replaced by  $\alpha_i$. Therefore, $|\tilde{\omega}|<1$. Hence the function $f$ is locally univalent and sense-preserving. Part $(2)$  follows similarly.
\end{proof}
\begin{theorem}\label{p2theom15}
For $i=1,2$ and $n\in\mathbb{N}-\left\lbrace1\right\rbrace,$ let the function $f_{\alpha_i,n}=h_{\alpha_i,n}+\overline{g}_{\alpha_i,n},$ be the normalized harmonic mapping satisfying $h_{\alpha,n}(z)+g_{\alpha,n}(z)=\big(z(1+z^2)(1+z^4)\dots(1+z^{2^n}+\alpha z^{2^{n-1}})/(1+z^{2{n+1}})\big)*\log{1/(1-z)}$, $\alpha_i\in[-2(\sqrt{2}-1),2(\sqrt{2}-1)] $. If $\omega_1(z)=-z^{2^{n-2}}$ and $\omega_2(z)=z^{2^{n-1}}$ are dilatations respectively of functions $f_{\alpha_1,n}$ and $f_{\alpha_2,n}$, then the  convex combination $f=tf_{\alpha_1,n}+(1-t)f_{\alpha_2,n}$, $0\leq t\leq 1$ belongs to $\mathcal{S}_H$ and is convex in the direction of imaginary axis, provided $\alpha_1\leq\alpha_2.$
\end{theorem}
\begin{proof}
 For $\alpha_1=\alpha_2$, the result is proved in Theorem \ref{p2theom12}, so we consider the case $\alpha_1<\alpha_2$. Also, for $t=0,1$, the result has been already shown in the discussion proceeding the  Theorem \ref{p2theom8}, so we will prove it for the case $0<t<1$. Let $\tilde{\omega}$ be the dilatation of the function $f$. Using $\omega_1(z)=-z^{2^{n-2}}$ and $\omega_2(z)=z^{2^{n-1}}$ in \eqref{p2eq11}, we get \[\tilde{\omega}(z)=\left(\frac{\splitfrac{-tz^{2^{n-2}}(1+z^{2^n}+\alpha_1 z^{2^{n-1}})(1+z^{2^{n-1}})}{+(1-t)z^{2^{n-1}}(1+z^{2^n}+\alpha_2 z^{2^{n-1}})(1-z^{2^{n-2}})}}{\splitfrac{t(1+z^{2^{n-1}})(1+z^{2^n}+\alpha_1 z^{2^{n-1}})}{+(1-t)(1-z^{2^{n-2}})(1+z^{2^n}+\alpha_2 z^{2^{n-1}})}}\right).\] Take  $z^{2^{n-2}}=w$, above equation gives
\begin{align*}
\tilde{\omega}(z)&=\frac{-tw(1+w^4+\alpha_1 w^2)(1+w^2)+(1-t)w^2(1+w^4+\alpha_2 w^2)(1-w)}{t(1+w^2)(1+w^4+\alpha_1 w^2)+(1-t)(1-w)(1+w^4+\alpha_2 w^2)},\\&=-w\frac{p(w)}{p^*(w)},\\\intertext{where}p(w)&=w^6+(t-1)w^5+(\alpha_2+(1+\alpha_1-\alpha_2)t)w^4+\alpha_2(t-1)w^3\\&\quad+(1+\alpha_1t)w^2+(t-1)w+t\\\intertext{and}p^*(w)&=w^6\overline{p}(1/\overline{w}).
\end{align*}
Therefore, if $w_1,w_2,w_3,w_4,w_5,w_6$ are zeros of $p(z)$, then $1/w_1,1/w_2,1/w_3,1/w_4,1/w_5,1/w_6$ are zeros the of the polynomial $p^*(w)$, and we can write $\tilde{\omega}$ as\[\tilde{\omega}(z)=-w\frac{(w-w_1)(w-w_2)(w-w_3)(w-w_4)(w-w_5)(w-w_6)}{(1-\overline{w}_1w)(1-\overline{w}_2w)(1-\overline{w}_3w)(1-\overline{w}_4w)(1-\overline{w}_5w)(1-\overline{w}_6w)}.\]Thus, to show that $|\tilde{\omega}(z)|<1$, it is enough to show $|w_i|<1$, $i=1,2,\dots,6.$\\ Consider the polynomial $q_1$ given by
  \begin{align*}
  q_1(w)&=\frac{p(w)-tp^*(w)}{w}\\&=(1-t^2)w^5-(1-t)^2w^4+(1-t)(\alpha_2+\alpha_1t)w^3-\alpha_2(1-t)^2w^2\\&\quad+(1-t)(1+(1+\alpha_1-\alpha_2)t)w-(1-t)^2\\&=(1-t)p_1(z), \intertext{where}p_1(z)&=(1+t)w^5-(1-t)w^4+(\alpha_2+\alpha_1t)w^3-\alpha_2(1-t)w^2\\&\quad+(1+(1+\alpha_1-\alpha_2)t)w-(1-t).
\end{align*}
Since $|t|<1$, by Cohn's rule the number of zeros of the polynomial  $p_1$ in  $\mathbb{D}$ is one less than that of the polynomial $p$.
Again, consider the polynomial $q_2$ given by
\begin{align*}
q_2(w)&=\frac{(1+t)p_1(w)+(1-t)p_1^*(w)}{w}\\&=4tw^4+t(1-t)(\alpha_1-\alpha_2)w^3+(\alpha_1+3\alpha_2+(\alpha_1-\alpha_2)t)tw^2\\&\quad+(\alpha_1-\alpha_2)t(1-t)w+(4+(\alpha_1-\alpha_2)(1+t))t\\&=tp_2(w),\\\intertext{where}p_2(w)&=4w^4+(1-t)(\alpha_1-\alpha_2)w^3+(\alpha_1+3\alpha_2+(\alpha_1-\alpha_2)t)w^2\\&\quad+(\alpha_1-\alpha_2)(1-t)w+4+(\alpha_1-\alpha_2)(1+t).
\end{align*}
Since $|1-t|<|1+t|$, by Cohn's rule the number of zeros of the polynomial $p_2$ in $\mathbb{D}$ is one less than that of the polynomial $p_1$. Also, consider the polynomial $q_3$ given by
\begin{align*}
q_3(w)&=\frac{4p_2(w)-(4+(\alpha_1-\alpha_2)(1+t))p^*_2(z)}{w}\\&=-(\alpha_1-\alpha_2)(1+t)\lbrace(8+(\alpha_1-\alpha_2)(1+t))w^3+(\alpha_1-\alpha_2)(1-t)w^2\\&\quad+(\alpha_1+3\alpha_2+(\alpha_1-\alpha_2)t)w+(\alpha_1-\alpha_2)(1-t)\rbrace\\&=-(\alpha_1-\alpha_2)(1+t)p_3(w),\\\intertext{where}p_3(w)&=(8+(\alpha_1-\alpha_2)(1+t))w^3+(\alpha_1-\alpha_2)(1-t)w^2\\&\quad+(\alpha_1+3\alpha_2+(\alpha_1-\alpha_2)t)w+(\alpha_1-\alpha_2)(1-t).
\end{align*}
 Since $|4+(\alpha_1-\alpha_2)(1+t)|<4$, by Cohn's rule the number of zeros of the polynomial $p_3$ in  $\mathbb{D}$ is one less than that of the polynomial $p_2$. Now, for the zeros of $p_3$  in $\mathbb{D}$, consider
 \begin{align*}
 |p_3(w)-(8&+(\alpha_1-\alpha_2)(1+t))w^3|\\&=|(\alpha_1-\alpha_2)(1-t)w^2+(\alpha_1+3\alpha_2+(\alpha_1-\alpha_2)t)w+(\alpha_1-\alpha_2)(1-t)|\\&\leq2|(1-t)(\alpha_1-\alpha_2)|+|\alpha_1+3\alpha_2+(\alpha_1-\alpha_2)t|\\&=2(1-t)(\alpha_2-\alpha_1)+|\alpha_1+3\alpha_2+(\alpha_1-\alpha_2)t|\\&<(8+(\alpha_1-\alpha_2)(1+t)).
 \end{align*}
 Thus, $|p_3(w)-(8+(\alpha_1-\alpha_2)(1+t))w^3|<|(8+(\alpha_1-\alpha_2)(1+t))w^3|$ on $|w|=1$, by Rouche's theorem all the three zeros of $p_3(w)$ lie in $\mathbb{D}$. Hence, we see that all the six zeros of the polynomial $p$ lie in $\mathbb{D}$. Therefore, $|\tilde{\omega}|<1$. The result now follows from Theorem \ref{p2theom10}.
\end{proof}
Upon looking at the Theorems \ref{p2theom3}, \ref{p2theom4} and \ref{p2theom5}, we can propose the following questions.
\begin{itemize}
\item[(1)] To find the values of $a\in(-1,1)$ and $\theta\in\mathbb{R}$, such that the result in the Theorem \ref{p2theom3} holds if we take $\omega(z)=\emph{e}^{\textit{i}\theta}(a-z^n)/(1-az^n),n=3,4,\dots.$ \\
 \item[(2)] To find the values of $a\in(-1,1)$ and $\theta\in\mathbb{R},$ such that the result in the Theorem \ref{p2theom4} holds if we take $\omega(z)=\emph{e}^{\textit{i}\theta}(a-z^n)/(1-az^n),n=3,4,\dots$.\\
\item[(3)] To find the values of $a\in(-1,1)$ and $\theta\in\mathbb{R},$ such that the result in the Theorem \ref{p2theom5} holds if we take $\omega(z)=\emph{e}^{\textit{i}\theta}(a-z)^n/(1-az)^n,n=3,4,\dots$.\end{itemize}


\begin{thebibliography}{00}
\bibitem{1}M. Dorff, M. Nowak\ and\ M. Wo\l oszkiewicz, Convolutions of harmonic convex mappings, Complex Var. Elliptic Equ. {\bf 57} (2012), no.~5, 489--503.
\bibitem{2}L. Li\ and\ S. Ponnusamy, Convolutions of slanted half-plane harmonic mappings, Analysis (Munich) {\bf 33} (2013), no.~2, 159--176. MR3082279
\bibitem{3}R. Kumar, M. Dorff, S. Gupta\ and\ S.  Singh,  Convolution properties of some harmonic mappings in the right half-plane, Bull. Malays. Math. Sci. Soc. 39 (2016), no. 1, 439--455.
\bibitem{4}Q. I. Rahman\ and\ G. Schmeisser, {\it Analytic theory of polynomials}, London Mathematical Society Monographs. New Series, 26, Oxford Univ. Press, Oxford, 2002.
\bibitem{5}W. Hengartner\ and\ G. Schober, On Schlicht mappings to domains convex in one direction, Comment. Math. Helv. {\bf 45} (1970), 303--314.
\bibitem{6}J. Clunie\ and\ T. Sheil-Small, Harmonic univalent functions, Ann. Acad. Sci. Fenn. Ser. A I Math. {\bf 9} (1984), 3--25.
\bibitem{7}Z.-G. Wang, Z.-H. Liu\ and\ Y.-C. Li, On the linear combinations of harmonic univalent mappings, J. Math. Anal. Appl. {\bf 400} (2013), no.~2, 452--459.
\bibitem{8}R. Kumar, S. Gupta\ and\ S. Singh, Linear combinations of univalent harmonic mappings convex in the direction of the imaginary axis, Bull. Malays. Math. Sci. Soc. {\bf 39} (2016), no.~2, 751--763.
\bibitem{9}M. Dorff, Convolutions of planar harmonic convex mappings, Complex Variables Theory Appl. {\bf 45} (2001), no.~3, 263--271.
\end{thebibliography}
\end{document}